\documentclass[12pt,reqno]{amsart}
\usepackage{hyperref}
\usepackage{cmap}                        
\usepackage[cp1251]{inputenc}            
\usepackage[english]{babel}
\usepackage[left=1in,right=1in,top=1in,bottom=1.5in]{geometry} 
\usepackage{amssymb,amsmath, amsthm, amscd,ifthen}
\usepackage{graphicx, xcolor}
\usepackage{bm}

\newtheorem*{thm*}{Theorem}
\newcommand{\ff}{{\mathcal F}}
\newcommand{\aaa}{{\mathcal A}}

\newtheorem*{cla*}{Claim}
\newcommand{\bb}{{\mathcal B}}

\newtheorem{thm}{Theorem}
\newtheorem{gypo}{Conjecture}
\newtheorem{opr}{Definition}
\newtheorem{lem}[thm]{Lemma}
\newtheorem{cla}[thm]{Claim}

\newtheorem{cor}[thm]{Corollary}
\date{}

\newtheorem{prop}[thm]{Proposition}
\newtheorem{obs}[thm]{Observation}

\DeclareMathOperator{\E}{\mathrm E}

\title{Two problems on matchings in set families -- in the footsteps of Erd\H os and Kleitman}
\author{Peter Frankl}\address{R\'enyi Institute, Budapest, Hungary; Email: {\tt peter.frankl@gmail.com}}

\author{Andrey Kupavskii}
\address{University of Birmingham and Moscow Institute of Physics and Technology; Email: {\tt kupavskii@ya.ru}.} \thanks{The research of the second author was supported by the grant RNF~16-11-10014.}

\date{}

\begin{document}
\maketitle
\begin{abstract}
The families $\ff_1,\ldots, \ff_s\subset 2^{[n]}$ are called \textit{$q$-dependent} if there are no pairwise disjoint $F_1\in \ff_1,\ldots, F_s\in\ff_s$ satisfying $|F_1\cup\ldots\cup F_s|\le q.$ We determine $\max |\ff_1|+\ldots +|\ff_s| $ for all values $n\ge q,s\ge 2$. The result provides a far-reaching generalization of an important classical result of Kleitman.

The well-known Erd\H os Matching Conjecture suggests the largest size of a family $\ff\subset {[n]\choose k}$ with no $s$ pairwise disjoint sets. After more than 50 years its full solution is still not in sight. In the present paper we provide a Hilton-Milner-type stability theorem for the Erd\H os Matching Conjecture in a relatively wide range, in particular, for $n\ge (2+o(1))sk$ with $o(1)$ depending on $s$ only. This is a considerable improvement of a classical result due to Bollob\'as, Daykin and Erd\H os.

We apply our results to advance in the following anti-Ramsey-type problem, proposed by \"Ozkahya and Young. Let  $ar(n,k,s)$ be the minimum number $x$ of colors such that in any coloring of the $k$-element subsets of $[n]$ with $x$ (non-empty) colors there is a \textit{rainbow matching} of size $s$, that is, $s$ sets of different colors that are pairwise disjoint. We prove a stability result for the problem, which allows to determine $ar(n,k,s)$ for all $k\ge 3$ and $n\ge sk+(s-1)(k-1).$ Some other consequences of our results are presented as well.

\end{abstract}
\section{Introduction}
Let $[n] := \{1,2,\ldots, n\}$ be the standard $n$-element set and $2^{[n]}$ its power set. A subset $\mathcal F\subset 2^{[n]}$ is called a \textit{family}. For $0\le k\le n$ let ${[n]\choose k}$ denote the family of all $k$-subsets of $[n]$.

For a family $\ff$, let $\nu(\ff)$ denote the maximum number of pairwise disjoint members of $\ff$. Note that $\nu(\ff)\le n$ holds, unless $\emptyset \in \ff$. The fundamental parameter $\nu(\ff)$ is called the \textit{independence number} or \textit{matching number} of $\ff$.

Let us introduce an analogous notion for several families.

\begin{opr} Suppose that $\ff_1,\ldots, \ff_s\subset 2^{[n]}$, where $2\le s\le n.$ We say that $\ff_1,\ldots, \ff_s$ are \textit{cross-dependent} if there is no choice of $F_1\in \ff_1,\ldots, F_s\in \ff_s$ such that $F_1,\ldots, F_s$ are pairwise disjoint.
\end{opr}

Note that $\nu(\ff)<s$ is equivalent to saying that $\ff_1, \ldots, \ff_s$, where $\ff_i:=\ff$ for all $i\in[s]$, are cross-dependent.\\

\noindent\textbf{Example.} Let $n = sm+s-\ell$ for some $\ell\in [s]$. Define
\begin{equation*}\tilde{\ff}_{i} := \begin{cases}\{F\subset[n]:|F|\ge m\},\ \ \ \ \ \ \ \ \ \ 1\le i<\ell, \\
 \{F\subset[n]:|F|\ge m+1\},\ \ \ \ \ \ell\le i\le s.\end{cases}
\end{equation*}
Then $\tilde{\ff}_{1},\ldots,\tilde{\ff}_{s}$ are easily seen to be cross-dependent.\\

One of the main results of the present paper is as follows.
\begin{thm}\label{thm2} Choose integers $s,m,\ell$ satisfying $s\ge 2$, $m\ge 0$ and $\ell\in [s]$. Set $n = sm+s-\ell$ and suppose that $\mathcal F_1,\ldots, \ff_s\subset 2^{[n]}$ are cross-dependent.  Then
\begin{equation}\label{eqcross}\sum_{i=1}^s|\mathcal F_i|\le (l-1){n\choose m}+s\sum_{t\ge m+1}{n\choose t}\ \ \ \Biggl[= \sum_{i=1}^s|\tilde{\ff}_i|\Biggr].\end{equation}
\end{thm}
Inequality (\ref{eqcross}) extends the following important classical result of Kleitman.
\begin{thm*}[Kleitman, \cite{Kl}] Let $s\ge 2$ be an integer and $\ff\subset 2^{[n]}$ a family satisfying $\nu(\ff)<s.$ Then for $n=s(m+1)-\ell$ we have
\begin{align}\label{eq002}|\ff|\le \frac{\ell-1}s{n\choose m}+\sum_{t\ge m+1}{n\choose t}.
\end{align}
\end{thm*}
In the case $n=s(m+1)-1$ the families $\tilde\ff_i$ from the example above are all the same, and thereby the bound (\ref{eq002}) is best possible. It is also best possible for $\ell=s$, as the following example due to Kleitman shows:
$$\mathcal K:=\bigl\{K\subset [sm]:|K|\ge m+1\bigr\}\cup{[sm-1]\choose m}.$$
(Note that ${sm-1\choose m} = \frac {s-1}s{sm\choose m}$.) In the case $s=2$ the bound (\ref{eq002}) reduces to $|\ff|\le 2^{n-1}$. This easy statement was proved already by Erd\H os, Ko and Rado \cite{EKR}.

Although (\ref{eq002}) is a beautiful result, it has no matching lower bound for  $n\not\equiv 0$ or $-1 (\mathrm{mod}\ s).$ For $s=3$ the exact answer for the only remaining residue class was given by Quinn \cite{Q}. Recently, we made some further progress \cite{FK7, FK8} and in particular completely resolved the case $n\equiv -2(\mathrm{mod}\ s)$.

Let us mention that, except the case $s=3$ and $\ell=2$, an extension of the methods used by Kleitman in \cite{Kl} is sufficient to prove Theorem \ref{thm2}. However, for that single case it does not seem to work. This forced us to find a very different proof. It is given in Section \ref{sec8}. We proved related results refining the method developed in the present paper, see \cite{FK13,FK8,FK10}.\\

Next we discuss a generalization of the notion of cross-dependence.
\begin{opr}\label{def1} Let $2\le s\le n$ and $q\in[n]$ be fixed integers. The families $\ff_1,\ldots, \ff_s\subset 2^{[n]}$ are called \textit{$q$-dependent} if there are no pairwise disjoint $F_1\in\ff_1,\ldots, F_s\in \ff_s$ satisfying $|F_1\cup\ldots\cup F_s|\le q$.
\end{opr}
 For $q=n$ the notion of $q$-dependence reduces to that of cross-dependence. Quite surprisingly, one can determine the exact maximum of $|\ff_1|+\ldots+|\ff_s|$ for $q$-dependent families $\ff_1,\ldots, \ff_s\subset 2^{[n]}$ and \textit{all} values of $n,q,s$.

Let $s\ge 2,\ m\ge 0$ and $\ell\in [s]$. If $n\ge q := sm+s-\ell$, then one can define
\begin{equation}\label{eq085}\tilde{\ff}_{i}^{n,q} := \begin{cases}\{F\subset[n]:|F|\ge m\},\ \ \ \ \ \ \ \ \ \ 1\le i<\ell, \\
 \{F\subset[n]:|F|\ge m+1\},\ \ \ \ \ \ell\le i\le s.\end{cases}
\end{equation}
In Section~\ref{sec8} we prove the following generalization of Theorem \ref{thm2}. It follows by induction from Theorem~\ref{thm2}, which serves as the base case.
\begin{thm}\label{thm7} Choose integers $s,m,\ell$ satisfying $s\ge 2,$ $m\ge 0$ and $\ell\in [s]$. Put $q = sm+s-\ell$ and suppose that $n\ge q$. If $\mathcal F_1,\ldots, \ff_s\subset 2^{[n]}$ are $q$-dependent, then
\begin{equation}\label{eq019} \sum_{i=1}^s|\ff_i|\le  (\ell-1){n\choose m}+s\sum_{t=m+1}^n{n\choose t}\ \ \ \Biggl[=\sum_{i=1}^s|\tilde{\ff}_i^{n,q}| \Biggr].\end{equation}
\end{thm}

\subsection{A Hilton-Milner-type result for the Erd\H os Matching Conjecture}

The Kleitman Theorem was motivated by a conjecture of Erd\H os (see \cite{Kl}). Erd\H os \cite{E} himself studied the uniform case, i.e., the families $\ff \subset {[n]\choose k}$. Let us make a formal definition.

\begin{opr}\label{def2} For positive integers $n,k,s$ satisfying $s\ge 2$, $n\ge ks$, define
$$e_k(n,s) := \max\Bigl\{|\mathcal F|: \ff\subset {[n]\choose k}, \nu(\mathcal F)<s\Bigr\}.$$
\end{opr}

Note that for $s = 2$ the quantity $e_k(n,s)$ was determined by Erd\H os, Ko and Rado.
\begin{thm*}[Erd\H os-Ko-Rado \cite{EKR}]
\begin{equation}\label{eq005} e_k(n,2) = {n-1\choose k-1} \ \ \ \ \ \text{for } \ \ \ n\ge 2k>0.
\end{equation}
\end{thm*}
\noindent The case $s\ge 3$ is much harder.
There are several natural examples of families $\aaa\subset{[n]\choose k}$ satisfying $\nu(\aaa) = s$ for $n\ge (s+1)k$. Following \cite{F3}, for each $i\in [k]$ let us define the families $\aaa_i^{(k)}(n,s):$
\begin{equation}\label{eq0011} \aaa_i^{(k)}(n,s) := \Bigl\{A\in {[n]\choose k}:\bigl|A\cap [(s+1)i-1]\bigr|\ge i\Bigr\}.\end{equation}
Note that $|\aaa_1^{(k)}(n,s)|={n\choose k}-{n-s\choose k}$.

\begin{gypo}[Erd\H os Matching Conjecture \cite{E}]\label{conj2} For $n\ge (s+1)k$ we have
\begin{equation}\label{eq008} e_k(n,s+1) = \max\bigl\{|\aaa_1^{(k)}(n,s)|,|\aaa_k^{(k)}(n,s)|\bigr\}.
\end{equation}
\end{gypo}
\noindent The conjecture (\ref{eq008}) is known to be true for $k\le 3$ (cf. \cite{EG}, \cite{LM} and \cite{F11}).
Improving earlier results of \cite{E}, \cite{BDE}, \cite{HLS} and \cite{FLM}, the first author \cite{F4} proved

\begin{equation}\label{eq009} e_k(n,s+1) = {n\choose k}-{n-s\choose k}  \ \ \ \ \ \ \text{ for } \ \ \ \ \ n\ge (2s+1)k-s.\end{equation}
In the forthcoming paper \cite{FK14}, we improve this bound for large $s$. We note that the conjecture is related to several questions in combinatorics, probability and computer science, cf. \cite{aletal, Zh16}.

In the case $s = 1$, corresponding to the Erd\H os-Ko-Rado theorem, one has a very useful stability theorem due to Hilton and Milner \cite{HM}. Below we discuss it and its natural generalization  to the case $s>1$.

Let us define the following families.
\begin{multline*}\mathcal H^{(k)}(n,s) := \Bigl\{H\in{[n]\choose k}: H\cap [s]\ne \emptyset\Bigr\} \cup\bigl\{[s+1,s+k]\bigr\}-\\ \Bigl\{H\in{[n]\choose k}: H\cap [s] = \{s\},H\cap[s+1,s+k]=\emptyset\Bigr\}.\end{multline*}
Note that $\nu(\mathcal H^{(k)}(n,s)) = s$ for $n\ge sk$ and
\begin{equation}\label{eq010} |\mathcal H^{(k)}(n,s)| = {n\choose k}-{n-s\choose k}+1-{n-s-k\choose k-1}.
\end{equation}
The {\it covering number} $\tau(\mathcal H)$ of a family $\mathcal H$ is the minimum of $|T|$ over all $T\subset [n]$ satisfying $T\cap H\ne \emptyset $ for all $H\in\mathcal H$. Recall the definition (\ref{eq0011}). The equality $\tau(\aaa_1^{(k)}(n,s)) = s$ for $n\ge k+s$ is obvious. At the same time, if $n\ge k+s$, then $\tau(\mathcal H^{(k)}(n,s)) = s+1$ and $\tau(\aaa_i^{(k)}(n,s))>s$ for each $i\ge 2$.

Let us make the following conjecture.

\begin{gypo}\label{conj3} Suppose that $\ff\subset {[n]\choose k}$ satisfies $\nu(\ff) = s, \tau(\ff)>s$. Then
\begin{equation}\label{eq011}|\ff|\le \max\Bigl\{\bigl\{|\aaa_i^{(k)}(n,s)|: i = 2,\ldots, k\bigr\},\ |\mathcal H^{(k)}(n,s)|\Bigr\}.\end{equation}
\end{gypo}
One can verify that  $|\mathcal A_i^{(k)}(n,s)|\le \max\bigl\{|\aaa_1^{(k)}(n,s)|,|\aaa_k^{(k)}(n,s)|\bigr\}$ for any $i\in k$. Modulo this verification, Conjecture~\ref{conj3} implies Conjecture~\ref{conj2}: indeed, we have $|\mathcal H^{(k)}(n,s)|\le |\aaa_1^{(k)}(n,s)|,$ and so the maximum on the right hand side of (\ref{eq011}) is at most the maximum on the right hand side of (\ref{eq008}). The Hilton-Milner theorem shows that (\ref{eq011}) is true for $s=1$.

\begin{thm*}[Hilton-Milner \cite{HM}] Suppose that $n> 2k$ and let $\ff\subset {[n]\choose k}$ be a family satisfying $\nu(\ff) = 1$ and $\tau(\ff)\ge 2.$ Then
\begin{equation*} |\ff|\le |\mathcal H^{(k)}(n,1)|.\end{equation*}
\end{thm*}
 We mention that for $n>2sk$ the maximum on the RHS of (\ref{eq011}) is attained on $|\mathcal H^{(k)}(n,s)|.$ For $n>2k^3s$, (\ref{eq011}) was verified by Bollob\'as, Daykin and Erd\H os \cite{BDE}.

Our second main result is the proof of (\ref{eq011}) in a much wider range.

\begin{thm}\label{thmhil} Suppose that $k\ge 3$ and either $n\ge (s+\max\{25,2s+2\})k$ or  $n\ge \bigl(2+o(1)\bigr)sk$, where $o(1)$ is with respect to $s\to \infty$. Then for any $\mathcal G\subset {[n]\choose k}$ with $\nu(\mathcal G)=s<\tau(\mathcal G)$ we have $|\mathcal G|\le |\mathcal H^{(k)}(n,s)|$.
\end{thm}

We prove Theorem \ref{thmhil} in Section \ref{sechil}. We note that stability results for the Erd\H os Matching Conjecture \cite{EKL}, as well as some important progress for a much more general class of Turan-type problems \cite{KLchv} were obtained recently. However, these results deal with the case $n>f(s)\cdot k$, where $f(s)$ is a fast growing function depending on $s$. Some other Hilton-Milner-related stability results were recently proven in \cite{KM} for $n>n_0(s,k)$.\\

Let us recall the method of \textit{left shifting} (or simply \textit{shifting}). For a given  pair of indices $1\le i<j\le n$ and a set $A \in 2^{[n]}$ we define the {\it $(i,j)$-shift} $S_{i,j}(A)$ of $A$ in the following way.
$$S_{i,j}(A) := \begin{cases}A \ \ \ \ \ \ \ \ \ \ \ \ \ \ \ \ \ \ \ \ \text{if } i\in A\ \ \text{or }\ \ \ j\notin A;\\(A-\{j\})\cup \{i\}\ \ \text{if } i\notin A\ \ \text{and }\ j\in A.
\end{cases}$$
Next, we define the {\it $(i,j)$-shift} $S_{i,j}(\mathcal F)$ of a family $\mathcal F\subset 2^{[n]}$:

$$S_{i,j}(\mathcal F) := \{S_{i,j}(A): A\in \mathcal F\}\cup \{A: A,S_{i,j}(A)\in \mathcal F\}.$$
We call a family $\ff$ \textit{shifted}, if $S_{i,j}(\ff) = \ff$ for all $1\le i<j\le n$.


\section{Proof of Theorem \ref{thmhil} }\label{sechil}


For $s = 1$ the theorem follows from the Hilton-Milner theorem, therefore we may assume that $s\ge 2$. Choose an integer $u$ such that \begin{equation}\label{eq184}  (u+s)(k-1)+s+1\le n < (u+s+1)(k-1)+s+1.\end{equation} Consider a family $\mathcal G$ satisfying the requirements of the theorem.

\subsection{The case of shifted $\mathcal G$}
First we prove Theorem \ref{thmhil} in the assumption that $\mathcal G$ is shifted.
Following \cite{F4}, we say that the families $\ff_1,\ldots, \ff_{s+1}$ are \textit{nested} if $\mathcal F_1\supset \mathcal F_2\supset\ldots\supset \mathcal F_{s+1}$.
The following lemma is a crucial tool for the proof and may be obtained by a straightforward modification of the proof of \cite[Theorem~3.1]{F4}:

\begin{lem}[Frankl \cite{F4}]\label{lem61}  Let $N\ge (u+s)(k-1)$ for some $u\in\mathbb Z$, $u\ge s+1$, and suppose that $\mathcal F_1,\ldots, \mathcal F_{s+1}\subset{[N]\choose k-1}$ are cross-dependent and nested. Then
\begin{equation}\label{eq112} |\mathcal F_1|+|\mathcal F_2|+\ldots +|\mathcal F_{s}|+u|\mathcal F_{s+1}|\le s{N\choose k-1}.\end{equation}
\end{lem}


We use the following notation. For a family $\mathcal G\subset 2^{[n]}$, all $p\in [n]$ and $Q\subset [p]$ define $$\mathcal G(Q,p) :=\big\{G\setminus Q: G\in \mathcal G, G\cap [p]=Q\big\}.$$
The first step of the proof of Theorem \ref{thmhil} is the following lemma.

\begin{lem}\label{lem62} Assume that $|\mathcal G|-|\mathcal G(\emptyset,s)|\le {n\choose k}-{n-s\choose k}-C$ for some $C>0$ and that $\nu(\mathcal G(\emptyset,s))=x$ for some $x\in [s]$. Then
\begin{equation}\label{eq128} |\mathcal G|\le  {n\choose k}-{n-s\choose k}-\frac{u-x-1}{u}C.\end{equation}
Moreover, if $\nu(\mathcal G(\emptyset,s))=\tau(\mathcal G(\emptyset,s))=1$, then
\begin{equation}\label{eq129} |\mathcal G|\le  {n\choose k}-{n-s\choose k}-\frac{u-1}{u}C.\end{equation}
\end{lem}
\begin{proof}Recall the definition of the {\it immediate shadow} of a family $\mathcal H$:
\begin{equation}\label{eq86}\partial \mathcal H := \bigl\{H:\exists H'\in \mathcal H, H\subset H', |H'\setminus H|=1\bigr\}.\end{equation}
 We have $\partial\mathcal G(\emptyset,s+1)\subset \mathcal G(\{s+1\},s+1)$ since $\mathcal G$ is shifted. In \cite[Theorem 1.2]{F4} the first author proved the inequality  $x|\partial \mathcal H|\ge |\mathcal H|$, valid for any $\mathcal H$ with $\nu(\mathcal H)\le x$. Due to the shiftedness of $\mathcal G$, we also have $\mathcal G(\emptyset,s+1)=\emptyset$ if $\tau(\mathcal G(\emptyset,s))=1$. Combining these three facts,  we get that
\begin{equation}\label{eq113}|\mathcal G(\emptyset,s+1)|\le x'|\mathcal G(\{s+1\},s+1)|,\end{equation}
where $x'=x$ if $\tau(\mathcal G(\emptyset,s))>1$ and $x'=0$ if $\tau(\mathcal G(\emptyset,s))=1$. Consequently, \begin{equation}\label{eq114}|\mathcal G(\emptyset,s)|\le (x'+1)|\mathcal G(\{s+1\},s+1)|.\end{equation}
Using (\ref{eq113}) and (\ref{eq112}), we have
\begin{multline*}\mathcal \sum_{i=1}^{s+1}|\mathcal G(\{i\},s+1)| + |G(\emptyset,s+1)| \le \sum_{i=1}^{s}|\mathcal G(\{i\},s+1)|+(x'+1)|\mathcal G(\{s+1\},s+1)|\le \\ \le s{n-s-1\choose k-1}-(u-x'-1)|\mathcal G(\{s+1\},s+1)|.\end{multline*}

For any $Q\subset [1,s+1]$, $|Q|\ge 2$, we have $|\aaa_1^{(k)}(n,s)(Q,s+1)| = {[s+2,n]\choose k-|Q|},$ and thus $|\mathcal G(Q,s+1)|\le |\mathcal \aaa_1^{(k)}(n,s)(Q,s+1)|$. We also have $\aaa_1^{(k)}(n,s)(\emptyset,s+1) = \emptyset$ and $\sum_{i=1}^{s+1}|\aaa_1^{(k)}(n,s)(\{i\},s+1)| = s{n-s-1\choose k-1}$.

Therefore, $|\aaa_1^{(k)}(n,s)|-|\mathcal G|\ge (u-x'-1)|\mathcal G(\{s+1\},s+1)|\overset{(\ref{eq114})}{\ge}\frac{u-x'-1}{x'+1} |\mathcal G(\emptyset,s)|.$ On the other hand, the inequality in the assumptions of the lemma may be formulated as $|\aaa_1^{(k)}(n,s)|-|\mathcal G|\ge C - |\mathcal G(\emptyset,s)|$. Adding these two inequalities (the second one multiplied by $\frac{u-x'-1}{x'+1}$), we obtain that $|\aaa_1^{(k)}(n,s)|-|\mathcal G|\ge \frac{u-x'-1}uC$.
\end{proof}

Now, to prove Theorem \ref{thmhil}, it is sufficient to obtain good bounds on $C$ from the formulation of Lemma \ref{lem62}. We do that in the next two propositions. We use the following simple observation:
\begin{obs}\label{obs3} If for some $C>0$, $S\subset [s]$ and $\mathcal B\subset {[s+1,n]\choose k-1}$ we have $\sum_{i\in S} \bigl|\mathcal G(\{i\},s)\cap \mathcal B\bigr|\le |S||\mathcal B|-C,$ then both $\sum_{i\in S} \bigl|\mathcal G(\{i\},s)\bigr|\le |S|{n-s\choose k-1}-C$ and $|\mathcal G|-|\mathcal G(\emptyset,s)|\le {n\choose k}-{n-s\choose k}-C$.
\end{obs}

We are going to use the next proposition and lemma for the case $\nu(\mathcal G(\emptyset,s))\ge 2$. Assume that $\mathcal G(\emptyset,s)$ contains $x$ pairwise disjoint sets $F_1,\ldots, F_x$ for some $x\in [s]$.
Put $\mathcal B_j:={[s+1,n]\setminus F_j\choose k-1}$.
\begin{prop}\label{prop19} Under the assumption above, choose a positive integer  $q$ and integers $0=:p_0< p_1<p_2<\ldots<p_q:=x$. Put $f:=\prod_{j=1}^q(p_j-p_{j-1})$. Then for $u\ge qf+\frac{q-1}{k-1}$ we have
 \begin{align}\label{eq171}&\sum_{i=1}^{s}\bigl|\mathcal G(\{i\},s)\bigr|\le s{n-s\choose k-1}-q\Bigl|\bigcap_{j=0}^{q-1}\Bigl(\bigcup_{z=p_j+1}^{p_{j+1}}\mathcal B_z\Bigr)\Bigr|.\end{align}
\end{prop}
\begin{proof}
For each $i\in[s]$ denote $$\mathcal I(\{i\},s) := \mathcal G(\{i\},s)\cap \bigcap_{j=0}^{q-1}\Bigl(\bigcup_{z=p_j+1}^{p_{j+1}}\mathcal B_z\Bigr).$$

Assume that $|\mathcal I(\{s-q+1\},s)|=y$. Then, since  $\mathcal I(\{i\},s)\supset\mathcal I(\{i+1\},s)$ for any $i\in [s-1]$, we have

\begin{equation}\label{eq172}\sum_{i=s-q+1}^s|\mathcal I(\{i\},s)|\le qy.\end{equation}
Applying Observation \ref{obs3} with $S = [s-q+1,s]$ and $\bb = \bigcap_{j=0}^{q-1}\Bigl(\bigcup_{z=p_j+1}^{p_{j+1}}\mathcal B_z\Bigr),$ we get

\begin{equation}\label{eq181}\sum_{i=s-q+1}^s \bigl|\mathcal G(\{i\},s)\bigr|\le q{n-s\choose k-1}-q\Bigl|\bigcap_{j=0}^{q-1}\Bigl(\bigcup_{z=p_j+1}^{p_{j+1}}\mathcal B_z\Bigr)\Bigr|+qy.\end{equation}

Note that $$\bigcap_{j=0}^{q-1}\Bigl(\bigcup_{z=p_j+1}^{p_{j+1}}\mathcal G(\{s-q+1\},s)\cap\mathcal B_z\Bigr) = \bigcup_{z_0=p_0+1}^{p_1}\bigcup_{z_1=p_1+1}^{p_2}\ldots \bigcup_{z_{q-1}=p_{q-1}+1}^{p_q}\Big(\bigcap_{j=0}^{q-1}\mathcal G(\{s-q+1\},s)\cap \mathcal B_{z_j}\Big).$$
Since $|\mathcal I(\{s-q+1\},s)|=y$, by the pigeon-hole principle from the equality above we infer that there exist a choice of $z_0'\in [p_0+1,p_1],\ldots, z_{q-1}'\in[p_{q-1}+1,p_q]$, such that
\begin{equation}\label{eq173} \Bigl|\mathcal G(\{s-q+1\},s)\cap \bigcap_{j=0}^{q-1}\mathcal B_{z_j'}\Bigr|\ge \frac yf.\end{equation}

 Next,  the families $\mathcal G_z(\{1\},s),\ldots, \mathcal G_z(\{s-q+1\},s)$, where $\mathcal G_z(\{i\},s):= \mathcal G(\{i\},s)\cap \bigcap_{j=0}^{q-1}\mathcal B_{z_j'},$ are cross-dependent and nested. Indeed, if the families are not cross-dependent, and there exist $G_1\in \mathcal G_z(\{1\},s),\ldots, G_{s-q+1}\in \mathcal G_z(\{s-q+1\},s)$ that are pairwise disjoint, then $G_1,\ldots, G_{s-q+1}, F_{z_0'},\ldots, F_{z_{q-1}'}$ form an $(s+1)$-matching in $\mathcal G$. Note that $\bigcap_{j=0}^{q-1}\mathcal B_{z_j} = {S'\choose k-1}$, where $S':=[s+1,n]\setminus (F_{z_0'}\cup\ldots\cup F_{z_{q-1}'})$, and $|S'|=(u+s)(k-1)+1-qk$.
 We have $u':=u-\frac {q-1}{k-1}\ge s-q+1$ since $u\ge s+1$. Moreover, $|S'|= (u'+s-q)(k-1)$. Thus, we may apply \eqref{eq112} with $\lfloor u'\rfloor,|S'|,s-q$ playing the roles of $u,N, s$, respectively. From (\ref{eq112}) and the inequality $u'=u-\frac {q-1}{k-1}\ge fq$ we get
\begin{equation}\label{eq174}|\mathcal G_z(\{1\},s)|+\ldots +|\mathcal G_z(\{s-q\},s)|+fq|\mathcal G_z(\{s-q+1\},s)|\le (s-q)\Bigl|\bigcap_{j=0}^{q-1}\mathcal B_{z_j}\Bigr|,\end{equation}
which, in view of (\ref{eq173}), gives us
\begin{equation}\label{eq175}\sum_{i=1}^{s-q}|\mathcal G_z(\{i\},s)|\le (s-q)\bigl|\bigcap_{j=0}^{q-1}\mathcal B_{z_j}\bigr| -qy.\end{equation}
Applying Observation \ref{obs3} with $S = [1,s-q]$ and $\bb = \bigcap_{j=0}^{q-1}\mathcal B_{z_j}$, we get
\begin{equation}\label{eq182} \sum_{i=1}^{s-q} \bigl|\mathcal G(\{i\},s)\bigr|\le (s-q){n-s\choose k-1}-qy.
\end{equation}
We get the statement of the proposition by summing up (\ref{eq181}) and (\ref{eq182}).
 \end{proof}

The main difficulty in using Proposition~\ref{prop19} directly is that it is very difficult to deal with sums/products of binomial coefficients that arise when writing down the subtrahend in \eqref{eq171} explicitly. The proof of the following lemma is a way to work around it. The lemma itself is an important technical ingredient in establishing  bounds on $u$ in Theorem \ref{thmhil}.

\begin{lem}\label{lem8} Assume that $\nu\bigl(\mathcal G(\emptyset,s)\bigr) \ge x$ for $x\in [s]$. Then  we have
 \begin{equation}\label{eq176}\sum_{i=1}^{s}\bigl|\mathcal G(\{i\},s)\bigr|\le s{n-s\choose k-1}-\gamma {n-k-s\choose k-1},\end{equation}
 where \ \ \ \ \ \ \ (i)\ \ $\gamma = \frac 43$ \  \, for $x=2$ \ \ \  and $u\ge 3$, \ \ \ \ \ \ \ \ \  (ii) $\gamma = \frac 32$ for $x=3$ \ and $u\ge 5$,\\ \phantom{esdfddd} \ \ \ \  (iii) $\gamma =\frac {16}9$ \  for $x=4,5$\ \  and $u\ge 9$,  \phantom{sdef}  \ \ \  (iv) $\gamma = 2 $ for $x\ge 6$\ \ and $u\ge 25$,\ \  \\ \phantom{eklrdfddfsfd}  (v)\, $\gamma = \Omega(x/\log_2^2 x)$ \ \ \ for $u\ge 2^x$.
\end{lem}
\begin{proof}
The logic of the proofs of all five statements is similar. We combine the bounds from Proposition \ref{prop19} with different parameters to get the bound of the form $\beta \sum_{i=1}^{s}\bigl|\mathcal G(\{i\},s)\bigr|\le \beta s{n-s\choose k-1}-\sum_{z=1}^{x}|\mathcal B_z|$ for the smallest possible $\beta$. Then the constant $\gamma$ from the statement of Lemma \ref{lem8} is defined as $\gamma := x/\beta$. Since $|\mathcal B_z| = {n-k-s\choose k-1}$ for any $z$, we get the statement, as long as we can guarantee the bounds on $\beta$ we claim. Therefore, we aim to find a linear combination of equations (\ref{eq171}) with coefficients $\beta_j$, which satisfies the following two conditions:
\begin{align}\label{eqa}&\text{a) the sum of the subtrahends in the RHS is at least } \sum_{z=1}^{x}|\mathcal B_z|,\\
 &b)\ \beta:=\sum\beta_j \text{ is as small as possible.}\end{align}

For any $S\subset [x]$ we introduce the following notation: $$\overline{\bigcap_{j\in S}}\mathcal B_{j}:= \Bigr(\bigcap_{j\in S}\mathcal B_{j}\Bigr)\setminus \Bigl(\bigcup_{j\in [x]\setminus S}\mathcal B_{j}\Bigr).$$
(Note that this definition depends on $x$, but the value of $x$ will be clear from the context.) Consider the following inclusion-exclusion-type decomposition:\begin{small}
\begin{equation}\label{eq177}\sum_{z=1}^{x}|\mathcal B_z| = \Big|\bigcup_{z=1}^{x}\mathcal B_z\Big|+\sum_{1\le z_1< z_2\le x}|\mathcal B_{z_1}\overline{\cap} \mathcal B_{z_2}|+2\sum_{1\le z_1< z_2<z_3\le x}\Big|\overline{\bigcap}_{j=1}^3\mathcal B_{z_j}\Big|+\ldots+(x-1)\Big|\overline{\bigcap}_{j=1}^x \mathcal B_j\Big|.\end{equation}\end{small}
The cardinalities of intersections in (\ref{eq177}) are determined by the number of intersecting families, but do not depend on the actual families that are intersecting. The number of summands of the form $\big|\overline{\bigcap}_{j=1}^{\ell}\mathcal B_{z_j}\big|$, multiplied by the coefficient $(\ell-1)$, is $(\ell-1){x\choose \ell}$ for any $2\le \ell\le x$. We call each of the families of the form $\overline{\bigcap}_{j=1}^{\ell}\mathcal B_{z_j}$ an \textit{$\ell$-intersection}. We denote its size by $\alpha_{\ell}$.
Putting $\eta_{\ell}:=(\ell -1){x\choose \ell}$ for each $2\le \ell\le x$, we can rewrite \eqref{eq177} as
\begin{equation}\label{eq177'}
  \sum_{z=1}^{x}|\mathcal B_z| = \Big|\bigcup_{z=1}^{x}\mathcal B_z\Big|+ \sum_{\ell=2}^x\eta_{\ell}\alpha_{\ell}.
\end{equation}

Each subtrahend in (\ref{eq171}) also admits a unique decomposition into $l$-intersections, analogous to (\ref{eq177}) (we will see some examples below). In the proof of each part of Lemma~\ref{lem8} we guarantee (\ref{eqa}) by finding the linear combination of \eqref{eq171} with different parameters, such that the resulting subtrahend has the form
\begin{equation}\label{eq178'}
  \Big|\bigcup_{z=1}^{x}\mathcal B_z\Big|+ \sum_{\ell=2}^x\eta_{\ell}'\alpha_{\ell},
\end{equation}
where $\eta'_{\ell}\ge \eta_{\ell}$ for each $2\le \ell\le x$.
In particular, the term $\big|\bigcup_{z=1}^{x}\mathcal B_z\big|$ is always ``covered'' by the subtrahend in (\ref{eq171}) with $q=1$, taken with coefficient $1$.
For each member of the linear combination we use the following notation: $\bigl[$parameters substituted in (\ref{eq171}); the  coefficient$\bigr]$.\vskip+0.1cm

 Since the proofs of (i)--(iv) are almost identical, we present the proofs of (i) and (iv) only. We start with (i). We sum up $\bigl[q=1$; coefficient 1$\bigr]$ with $\bigl[q=2$, $p_1=1$; coefficient $\frac 12\bigr]$. We note that for the latter member of the linear combination \eqref{eq171} gives
 $$\sum_{i=1}^{s}\bigl|\mathcal G(\{i\},s)\bigr|\le s{n-s\choose k-1}-2\Bigl|\mathcal B_1\cap \mathcal B_2\Bigr|.$$
 Returning to the linear combination, we get the inequality
\begin{equation*}\frac 32\Bigl|\bigcup_{i=1}^{s}\mathcal G(\{i\},s)\Bigr|\le \frac 32s{n-s\choose k-1}-\bigl|\mathcal B_1\cup\mathcal B_2\bigr|-\bigl|\mathcal B_1\cap\mathcal B_2\bigr| = \frac 32s{n-s\choose k-1}-|\mathcal B_1|-|\mathcal B_2|.\end{equation*}
The condition on $u$, imposed by the application of (\ref{eq171}), is satisfied for $u\ge 3$. It is clear that $\gamma =\frac 2{3/2}=\frac 43$ for this linear combination. This concludes the proof of (i).\vskip+0.1cm

The proof of (iv) is more cumbersome. It is sufficient to verify (iv) for $x=6$.
Take the following linear combination: $\bigl[q=1$; coefficient $1\bigr]$, $\bigl[q=2,\ p_1=3$; coefficient $\frac 32\bigr]$, $\bigl[q=3,\ p_1=2,\ p_2=4$; coefficient $\frac 13\bigl]$, and $\bigl[q=6,\ p_i=i$ for $i=0,\ldots,6$; coefficient $\frac 16\bigr]$. First, it is clear that for this combination we have $\beta=1+\frac 32+\frac 13+\frac 16= 3$, and thus $\gamma = x/\beta = 2$. Moreover, it is easy to see that the condition on $u$, imposed by the applications of (\ref{eq171}), is $u\ge 25$ and comes from the third summand.  Therefore, we are left to verify that the corresponding $\eta'_{\ell}$ satisfy $\eta'_{\ell}\ge \eta_{\ell}$ for each $2\le \ell \le 6$.

For $x=6$, the value of $\eta_{\ell}$ is $15, 40,  45, 24, 5$ for $\ell=2,\ldots, 6,$ respectively. Next, we find the values of $\eta'_{\ell}$. As it is not difficult to check, the subtrahends in \eqref{eq171} with $q=2,p_1=3$, $q=3,p_1=2,p_2=4$ and $q=6, p_i=i$ for $i=0,\ldots,6$, respectively, have the form
\begin{alignat}{3}
  \label{eqstupid1}(q=2)\ \ \ \  & 18\alpha_{2}+&&36 \alpha_3+ 30\alpha_4+12\alpha_5+&&2\alpha_6;\\
  \label{eqstupid2}(q=3)\ \ \ \ &  &&24 \alpha_3+ 36\alpha_4+18\alpha_5+&&3\alpha_6;\\
  \label{eqstupid3}(q=6) \ \ \ \ & && &&6\alpha_6.
\end{alignat}
Let us elaborate on how do we obtain \eqref{eqstupid1}. For $2\le \ell\le  6,$ put $\mathcal C_{\ell}:=\big\{C\in {[6]\choose \ell}: |C\cap [3]|\ge 1,|C\cap [4,6]|\ge 1\big\}$ and $\mathcal C:=\bigcup_{\ell=2}^6\mathcal C_{\ell}$.   One can easily check that $|C_{\ell}|$ equals $9,18,15,6,1$ for $\ell=2,\ldots, 6,$ respectively. The subtrahend in \eqref{eq171} has the form
\begin{equation*}2|(\mathcal B_1\cup \mathcal B_2\cup \mathcal B_3)\cap (\mathcal B_4\cup \mathcal B_5\cup \mathcal B_6)|=2\sum_{C\in\mathcal C}\Big|\overline{\bigcap_{i\in C}}\mathcal B_i\Big|=
18\alpha_2+36\alpha_3+30\alpha_4+12\alpha_5+2\alpha_6.\end{equation*}
Summing up \eqref{eqstupid1}, \eqref{eqstupid2}, \eqref{eqstupid3} with coefficients $\frac 32,\ \frac 13,\ \frac 16$, respectively, we get the expression $27\alpha_2+62\alpha_3+57\alpha_4+24\alpha_5+5\alpha_6,$ which is bigger than $\sum_{\ell=2}^{6}\eta_{\ell}\alpha_{\ell}$, as claimed. The proof of (iv) is complete. (We note that the choice of the coefficients, and thus the bound on $\gamma$, is clearly not optimal. It is however sufficient for our purposes.)\vskip+0.1cm

The proof of (v) is the most technical. Since we can always replace $x$ by any smaller positive integer, we assume for simplicity that $x = 2^r$ for some positive integer $r$, at the expense of factor $2$ that goes into the $\Omega$-notation. Consider the following linear combination:
$$M_j:=\bigl[q(j) = 2^j,\ p_i(j) = ix/q(j)\text{ for }i=0,\ldots, q(j);\text{ coefficient }4(j+2)\bigr],\text{ where }j = 0,\ldots, r.$$

First we verify that the linear combination above has enough $\ell$-intersections for each $\ell$. The union of all $\mathcal B_z$ is given by $M_0$. For larger $\ell$ we need to do an auxiliary calculation.

Let us determine, how many different $\ell$-intersections are contained in the family \begin{equation}\label{eq183}\bigcap_{i=0}^{q(j)-1}\Bigl( \bigcup_{z=p_i(j)+1}^{p_{i+1}(j)}\mathcal B_z\Bigr).\end{equation}
For $\ell\ge q(j)$, the $\ell$-subsets of $[x]$ corresponding to the $\ell$-intersections contained in the family above form the family $\mathcal C_{\ell}(j):=\big\{C\subset{[x]\choose \ell}: C\cap [p_i(j)+1,p_{i+1}(j)]\ne \emptyset\text{ for all }i=0,\ldots, q(j)-1\big\}$ (cf. the considerations after \eqref{eqstupid3}).
 We can bound the number of such $\ell$-intersections from below by ${x\choose \ell}\Big(1-q(j)\bigl(\frac{q(j)-1}{q(j)}\bigr)^{\ell}\Big)\ge {x\choose \ell}\big(1-q(j)e^{-\ell/q(j)}\big)$. For $\ell\ge q(j)\log_2 (2q(j))$ this is at more than $\frac 12{x\choose \ell}$. We conclude that the family (\ref{eq183}) contains more than $\frac 12{n\choose \ell}$ different $\ell$-intersections for $\ell\ge q(j)\log_2 (2q(j)) = 2^j(j+1)$.

Since the subtrahend in $M_j$ is the size of the family (\ref{eq183}) multiplied by the factor $2^{j+2}(j+2)$, we conclude that $M_j$ contributes more than $2^{j+1}(j+2){n\choose l}$ to  $\eta'_{\ell}$ for  $\ell\ge 2^j(j+1)$  (see \eqref{eq178'} for the definition of $\eta'_{\ell}$).
Next, for each $\ell\ge 2$, find the largest integer $j$ such that  $2^j(j+1)\le \ell$.
It is clear that $2^{j+1}(j+2)> \ell$. As we have shown above, $M_j$ contributes at least $2^{j+1}(j+2){n\choose l}>\ell{n\choose \ell}>\eta_{\ell}$ to the coefficient in front of $\alpha_l$. Thus, $\gamma_{\ell}'\ge \gamma_{\ell}$. One may also note that for $\ell=2,3$ the largest $j$  as defined above is $0$, and then the situation is slightly different since we have used the term $M_0$ to ``cover'' the union in \eqref{eq178'}. However, the ``unused'' coefficient for $M_0$ is $7$, and it is straightforward to see that $M_0$  contributes $7{n\choose \ell}$ to $\eta'_{\ell}$, which is clearly sufficient for $\ell=2,3$.

Second, we calculate the sum $\beta$ of the coefficients of $M_j$. We have
$$\beta =\sum_{i=0}^{r} 4(i+2) \le 4{r+3\choose 2}=  O(\log^2 x).$$
Thus, $\gamma = x/\beta = \Omega(x/\log^2 x)$.

Finally, we verify that the condition imposed on $u$ is sufficient for  the applications of (\ref{eq171}) we used. For $M_j$ the restriction is satisfied for $u\ge 2^j\bigl(2^{r-j}\bigr)^{2^j}+2^j=2^{j+(r-j)2^j}+2^j.$ This expression is clearly maximized when $j=r-1$, and in that case the inequality is $u\ge 2^{r-1+2^{r-1}}+2^{r-1}$. The latter expression is smaller than $2^{2^r}$ for any $r\ge 1$. Thus, the condition $u\ge 2^{x}\ge 2^{2^r}$ is sufficient.
\end{proof}

In the case $\nu(\mathcal G(\emptyset, s)) = 1$ we need a proposition which is more fine-grained than Proposition \ref{prop19}.
For each $j \in [k+1]$ put $D_j := [s+1,s+k+1]\setminus\{s+j\}$ and define the families $\mathcal C_j:={E_j\choose k-1}$, where $E_j:=[s+1,\ldots,n]\setminus D_j$ and $|E_j|\ge (u+s)(k-1)+1-k=(u+s-1)(k-1)$.

\begin{prop}\label{prop10} Assume that $\nu(\mathcal G(\emptyset,s)) = 1$ and put $v := \max\{1, k+2-u\}$.\begin{itemize}
\item[(i)] If $\tau(\mathcal G(\emptyset,s))>1$, then
\begin{equation}\label{eq126}\sum_{i=1}^{s}\bigl|\mathcal G(\{i\},s)\bigr|\le s{n-s\choose k-1}-\Big|\bigcup_{j=v}^{k+1}\mathcal C_j\Big|.\end{equation}
\item[(ii)] If $\tau(\mathcal G(\emptyset,s))=1$ and for some integer $t\in [v,k]$ we have $|\mathcal G(\emptyset,s)|>{n-s-t\choose k-t}$, then
\begin{equation}\label{eq127}\sum_{i=1}^{s}\bigl|\mathcal G(\{i\},s)\bigr|\le s{n-s\choose k-1}-\Big|\bigcup_{j=t}^{k+1}\mathcal C_j\Big|.\end{equation}\end{itemize}
\end{prop}
\begin{proof} (i)\ \ Since $\mathcal G(\emptyset,s)$ is shifted and $\tau(\mathcal G(\emptyset,s))>1$, the set $D_{1}$ is contained in $\mathcal G(\emptyset,s)$. Then, by shiftedness, for each $j\in[k+1]$ $D_j$ is contained in $\mathcal G(\emptyset,s)$. Arguing as in the proof of Proposition \ref{prop19}, let $\bigl|\bigl(\bigcup_{j=v}^{k+1} \mathcal C_j\bigr)\cap \mathcal G(\{s\},s)\bigr| =y$. Then there is an index $j\in[v, k+1]$ such that $\bigl|\mathcal C_j\cap \mathcal G(\{s\},s)\bigr| \ge \frac y{k+2-v}\ge \frac yu.$ The rest of the argument is very similar to the argument in Proposition~\ref{prop19}, which we reproduce for completeness.

Put $\mathcal G_j(\{i\},s):= \mathcal G(\{i\},s)\cap \mathcal C_j.$ The families $\mathcal G_j(\{1\},s),\ldots, \mathcal G_j(\{s\},s)\subset{E_j\choose k-1}$ are cross-dependent and nested.  Recall that $|E_j|\ge(u+s-1)(k-1)$. From (\ref{eq112}) we get
\begin{equation*}|\mathcal G_j(\{1\},s)|+\ldots +|\mathcal G_j(\{s-1\},s)|+u|\mathcal G_j(\{s\},s)|\le (s-1){|E_j|\choose k-1},\end{equation*}
which, in view of $|\mathcal G_j(\{s\},s)|\ge \frac yu$, gives us
\begin{equation*}\Bigl|\Bigl(\bigcup_{j=v}^{k+1} \mathcal C_j\Bigr)\cap \mathcal G(\{s\},s)\Bigr|+\sum_{i=1}^{s-1}|\mathcal G_j(\{i\},s)|\le (s-1){|E_j|\choose k-1}.\end{equation*}
Applying Observation \ref{obs3} with $S = [1,s-1]$ and $\bb = \bigcup_{j=v}^{k+1}{E_j\choose k-1}$, we  get (i).

\vskip+0.1cm

(ii) \ \ Similarly, since $\mathcal G(\emptyset,s)$ is shifted and $|\mathcal G(\emptyset,s)|>{n-s-t\choose k-t}$, the set $D_t$ must be contained in $\mathcal G(\emptyset,s)$. Therefore, each $D_j$ for $j\in [t,s+1]$ is contained in $\mathcal G(\emptyset,s)$, and we conclude as before.
\end{proof}

We go on to the proof of Theorem \ref{thmhil}.
In the case  $|\mathcal G(\emptyset,s)|=1$ we get exactly the bound stated in the theorem since $|\mathcal G| = |\mathcal G|-|\mathcal G(\emptyset,s)|+1 \le {n\choose k}-{n-s\choose k}-|\mathcal B_{1}|+1$. Thus, for the rest of the proof we assume that $|\mathcal G(\emptyset,s)|>1$.\\


\textbf{Case 1}. $\bm{\nu(\mathcal G(\emptyset,s))=\tau(\mathcal G(\emptyset,s))=1}$. If $1<|\mathcal G(\emptyset,s)|\le {n-s-k+1\choose 1} =n-s-k+1$ then by Proposition~\ref{prop10}, (ii) (with $t=k$) we have  $C\ge |\mathcal C_k\cup\mathcal C_{k+1}| = {n-k-s\choose k-1}+{n-k-s-1\choose k-2}$ for $C$ from Lemma \ref{lem62}. For $k\ge 4$ and $s\ge 2$ we have $${n-s-k-1\choose k-2}\ge n-s-k\ge|\mathcal G(\emptyset,s)|-1,$$ thus the theorem holds in this case. For $k=3$ one can see that $C\ge {n-s-2\choose 2}$ if $|\mathcal G(\emptyset,s)|\ge 4$,  and thus we have $C-{n-s-3\choose 2}=n-s-3\ge \mathcal G(\emptyset,s)-1$.


If $|\mathcal G(\emptyset,s)|>n-s-k+1$ then we use the following bound:
\begin{align} C\ge \Big|\bigcup_{j=k-1}^{k+1}\mathcal C_j\Big| = &{n-k-s\choose k-1}+2{n-k-s-1\choose k-2} = \notag \\ \label{eq131}= &\Bigl(1+\frac{2(k-1)}{n-k-s}\Bigr){n-k-s\choose k-1} \overset{(\ref{eq184})}\ge   \Bigl(1+\frac{2}{u+s}\Bigr){n-k-s\choose k-1}.\end{align}
The last expression is at least $\frac{u}{u-1}{n-k-s\choose k-1}$ if $\frac{u+s+2}{u+s}\ge\frac u{u-1}$, which holds for $u\ge s+2$. Since $u\ge s+2$, we can apply (\ref{eq129}) and conclude that Theorem \ref{thmhil} holds in this case.\\

\textbf{Case 2.} $\bm{\nu(\mathcal G(\emptyset,s))=1<\tau(\mathcal G(\emptyset,s))}$. Analogously to  (\ref{eq131}), Proposition \ref{prop10} implies that
$$C\ge \Big|\bigcup_{j=z}^{k+1}\mathcal C_j\Big|=\Bigl(1+\frac{\min\{k+1,u\}}{u+s-1}\Bigr){n-k-s\choose k-1}.$$
 The inequality $\frac{u+s+\min\{k+1,u\}}{u+s}\ge\frac u{u-2}$ holds for $u\ge s+4$ and $k\ge 3$, and we can apply (\ref{eq128}). \\

\textbf{Case 3.} $\bm{\nu(\mathcal G(\emptyset,s))=x\ge 2}$. We make use of Lemma \ref{lem8}. We are done in this case as long as, in terms of Lemma \ref{lem8},
\begin{equation}\label{eq180}\gamma\cdot\frac{u-x-1}u\ge 1.\end{equation}
Using Lemma \ref{lem8} (i)--(iv), one can see that \eqref{eq180} holds provided $u\ge \max\{25, 2s+2\}$. Indeed, let us verify this technical claim. It is clearly sufficient to verify it for $u=\max\{25, 2s+2\}$.  \begin{itemize}
\item If $x=2,$ then $\gamma = \frac 43$ and the left hand side of (\ref{eq180}) is at least $\frac 43\cdot \frac{22}{25}>1$.
\item If $x =3$, then the LHS is at least $\frac 32\cdot\frac{21}{25}> 1$.
\item If $x=4$, then the LHS is at least $\frac {16}9\cdot\frac{20}{25}>1.$
\item If $x=5$, then the LHS is at least $\frac {16}9\cdot\frac{19}{25}>1.$
\item If $x\ge 6$, then, using $s\ge x$, the LHS is at least $2\cdot \frac{x+1}{2x+2}=1.$\end{itemize}
To conclude this case, we remark that the inequalities $u\ge \max\{25,2s+2\},\ n\ge (s+u)(k-1)+s+1$ and $k\ge 3$ are sufficient for all the considerations above to work.\\

Next, using the fifth statement from Lemma \ref{lem8}, we obtain that (\ref{eq180}) is satisfied for $u = s+o(s)$. Indeed, take sufficiently large $s$ and put $u = s+\delta \frac{s (\log\log s)^2}{\log s}$ with some $\delta>0$ that will be determined later. If  $2x+25 \le s$, then the argument given in the previous paragraph shows that the condition $u\ge s$ is sufficient. Thus, we may assume that $x \ge (s-25)/2\ge \log_2 s$, where the second inequality holds  for all sufficiently large $s$. Then, applying the fifth point of Lemma \ref{lem8} with $x = \log_2 s$ (note that $u> s= 2^{x}$, and thus the condition of Lemma~\ref{lem8} are satisfied), we get
$$\gamma\cdot\frac{u-x-1}u \ge \Omega\Bigl(\frac {\log s}{(\log\log s)^2}\Bigr)\cdot \frac{u-s-1}u =  \Omega\Bigl(\frac {\log s}{(\log\log s)^2}\Bigr)\cdot \frac{\delta \frac{s (\log\log s)^2}{\log s}}s>1,$$
if $\delta$ is sufficiently large.
The proof of Theorem \ref{thmhil} for shifted families is complete.


\subsection{The case of not shifted $\mathcal G$}
Consider a family $\mathcal G$ satisfying the requirements of the theorem. Since the property $\tau(\mathcal G)>s$ is not necessarily maintained by $(i,j)$-shifts, we cannot assume that the family $\mathcal G$ is shifted right away. However,  each $(i,j)$-shift for $1\le i<j\le n$, decreases $\tau(\mathcal G)$ by at most 1.  Perform the $(i,j)$-shifts ($1\le i<j\le n$) one by one until either $\mathcal G$ becomes shifted or $\tau(\mathcal G) = s+1$. In the former case we fall into the situation of the previous subsection.

Now suppose that $\tau(\mathcal G) = s+1$ and that each set from $\mathcal G$ intersects $[s+1]$. Then each family $\mathcal G(\{i\},s+1)$ for $i\in[s+1]$ is nonempty. Make the family $\mathcal G$ shifted in coordinates $s+2,\ldots, n$ by performing all possible $(i,j)$-shifts with $s+2\le i<j\le n$. Denote the new family by $\mathcal G$ again. Since  shifts do not increase the matching number, we have $\nu(\mathcal G)\le s$ and $\tau(\mathcal G)\le s+1$. Each family $\mathcal G(\{i\},s+1)$ contains the set $[s+2,s+k]$.

Next, perform all possible shifts on coordinates $1,\ldots, s+1$ and denote the resulting family by $\mathcal G'$. We have $|\mathcal G'|=|\mathcal G|, \nu(\mathcal G')\le s$, and, most importantly, \begin{equation}\label{nonempty}\mathcal G'(\{1\},s+1), \ldots \mathcal G'(\{s+1\},s+1) \text{ are nested and non-empty.}\end{equation} These families are non-empty due to the fact that each of them contained the set $[s+2,s+k]$ before the shifts on $[s+1]$.

We can actually apply the proof from the previous subsection to $\mathcal G'$. Indeed, the main consequence of the shiftedness we were using is \eqref{nonempty}. The other consequence of the shiftedness was the bound (\ref{eq113}), which automatically holds in this case since every set from $\mathcal G'$ intersects $[s+1]$ and thus $\mathcal G'(\emptyset,s+1)$ is empty. The proof of Theorem \ref{thmhil} is complete.


\section{Proof of Theorems \ref{thm2} and \ref{thm7}}\label{sec8}

\subsubsection*{Proof of Theorem \ref{thm2}}
Take $s$ cross-dependent families $\ff_1,\ldots, \ff_s$.
For $s=2$ the bound (\ref{eqcross}) states that $|\mathcal \ff_1|+|\mathcal \ff_2|\le 2^{[n]}$, which follows from the following trivial observation: if $A\in \mathcal \ff_1$ then $[n]\setminus A\notin \ff_2$. Thus, we may assume that $s\ge 3$. Also, the case of $m=0$ is very easy to verify for any $s$, and so we assume that $m\ge 1$.

 Put $n := s(m+1)-\ell$ for the rest of this section. Recall that $\mathcal F$ is called an \textit{up-set} if for any $F\in \mathcal F$ all sets that contain $F$ are also in $\mathcal F$.
When dealing with cross-dependent and $q$-dependent families, we may restrict our attention to the families that are up-sets and shifted (e.g. see \cite[Claim 17]{FK7}), which we assume for the rest of the section.

Let us first treat the case $\ell=1$. This is the easiest case, and it will provide the reader with a good overview of the technique we use. Take $s$ pairwise disjoint sets $H_1,\ldots,H_s$ of size $m$ at random. To simplify notation, assume that the $s-1$ elements of $[n]\setminus\bigcup_{i=1}^sH_i$ form the set $[s-1]$.


For each $i\in[s]$ let $\emptyset =: H_i^{(0)}\subsetneq\ldots\subsetneq H_i^{(m)}:=H_i$ be a randomly chosen full chain in $H_i.$
For all $i \in [s]$, $0\le j\le m$ and $S\subset [s-1]$ define the random variables $\beta_i^{(j)}$ and $\beta_i(S)$:
\begin{equation}\label{eq50v1}\beta_i^{(j)} =\begin{cases}1\ \ \text{if }\ H_i^{(j)}\in \ff_i,\\0\ \ \text{if }\ H_i^{(j)}\notin \ff_i;\end{cases} \ \ \ \ \ \ \ \beta_i(S) =\begin{cases}1\ \ \text{if }\ H_i\cup S\in \ff_i,\\0\ \ \text{if }\ H_i\cup S\notin \ff_i.\end{cases}\end{equation}
Note that $S$ may be the empty set and that $\beta^{(m)}_i=\beta_i(\emptyset)$. The cross-dependence of $\ff_1,\ldots, \ff_s$  implies
\begin{equation}\label{eq45}\beta_{1}(S_1)\beta_{2}(S_2)\cdot\ldots\cdot\beta_{s}(S_s)=0\ \ \ \ \text{whenever}\ \ \ S_1,\ldots,S_s\subset[s-1] \ \ \text{are pairwise disjoint.}\end{equation}
For all $0\le j\le m$ and $S\subset [s-1]$ the expectations  $\E[\beta_i^{(j)}]$ and $\E[\beta_i(S)]$ satisfy
\begin{equation}\label{eq46v1}\E[\beta_i^{(j)}]=\frac{\bigl|\ff_i\cap{[n]\choose j}\bigr|}{{n\choose j}},\ \ \ \ \ \E[\beta_i(S)]=\frac{\bigl|\ff_i\cap{[n]\choose m+|S|}\bigr|}{{n\choose m+|S|}}.\end{equation}

Our aim is to prove the following statement.
\begin{lem}\label{lem11v1} For every choice of $H_1,\ldots, H_s$ and the full chains one has
\begin{equation}\label{eq47v1}\sum_{i=1}^s\Biggl[\sum_{j=0}^{m}{n\choose j}\beta_i^{(j)}+\sum_{S\in {[s-1]\choose 1}\cup{[s-1]\choose 2}}\frac{{n\choose m+1}}{{s-1\choose |S|}}\beta_i(S)\Biggr]\le s{n\choose m+1}+s{n\choose m+2},
\end{equation}
where $\beta_i^{(j)}$ and $\beta_i(S)$ are as defined in \eqref{eq50v1}.
\end{lem}
Replacing the left hand side of (\ref{eq47v1}) with its expected value, it is straightforward to see that (\ref{eq46v1}) implies $$\sum_{i=1}^s\sum_{j=0}^{m+2}\Bigl|\ff_i\cap{[n]\choose j}\Bigr|\le s{n\choose m+1}+s{n\choose m+2},$$
from which (\ref{eqcross}) follows.

\begin{proof}[Proof of Lemma~\ref{lem11v1}]
  If for all $i\in[s]$ we have $\beta_i^{(m)}=0$ then we are done. Assume that $\sum_{i=1}^s \beta_i^{(m)}=p$ for some $p\ge 1$. Clearly, $p\le s-1$ since $\ff_1,\ldots, \ff_s$ are cross-dependent. W.l.o.g., we assume that $\beta_i^{(m)}=1$ if and only if $i\in[p]$.

  If $p\le s-2$, then we proceed as follows. We have $\sum_{i=p+1}^s\beta_{i}(\{k_i\})\le s-p-1$ for any distinct $k_{p+1},\ldots, k_s\in [s-1]$. Averaging over the choice of $k_{p+1},\ldots, k_s$, it is easy to see that at least a $1/(s-p)$-fraction of the pairs $(i,k)$, where $p+1\le i\le s$ and $k\in [s-1]$, satisfies $\beta_i(\{k\})=0$. In other words, there are at least $s-1$ such pairs. Consequently, we get that the sum on the right hand side of \eqref{eq47v1} is at most $p\sum_{j=1}^{m}{n\choose j}+(s-1){n\choose m+1}+s{n\choose m+2}$ and the difference between the RHS and the LHS of \eqref{eq47v1} is at least $D:={n\choose m+1}-p\sum_{j=1}^m{n\choose j}$. It is easy to check that for $n=sm+s-1$ we have ${n\choose m+1}=(s-1){n\choose m}$, and thus $D\ge {n\choose m}-p\sum_{j=1}^{m-1}{n\choose j}$.\vskip+0.1cm

  For $n=sm+s-\ell$ with $\ell\in[s]$ and $m\ge 1$ the following holds.
\begin{equation}\label{eq53}\frac{{n\choose m-j-1}}{{n\choose m-j}} = \frac{m-j}{n-m+j+1}\le \frac{m-1}{(s-1)m} \ \ \ \ \ \ \  \text{for } \ \ \ \ j\ge 1.\end{equation}
Using (\ref{eq53}) we may obtain that $\sum_{j=1}^m p{n\choose m-j}$ is at most
\begin{multline}\label{eq55}\sum_{j=1}^m (s-2){n\choose m-j}\overset{\eqref{eq53}}{<} (s-2)\Big(1+\sum_{i=1}^{\infty}\frac{m-1}{m(s-1)^i}\Big){n\choose m-1} = \\(s-2)\Big(1+\frac{m-1}{m(s-2)}\Big){n\choose m-1}= \frac{\big(s-1-\frac 1m\big)m}{(s-1)m+s-\ell+1}{n\choose m}.\end{multline}
Using \eqref{eq55}, it is easy to conclude that $D>0$ in this case.

Assume now that $p=s-1$. Then $\beta_s(\{k\})=0$ for any $k\in[s-1]$. Then the LHS of \eqref{eq47v1} is at most $(s-1)\sum_{j=1}^{m+2}{n\choose j}$, and the difference $D$ between the RHS and the LHS of \eqref{eq47v1} satisfies $D\ge {n\choose m+2}+{n\choose m+1}-(s-1)\sum_{j=1}^{m}{n\choose j}$. Using the calculations from the previous case, we get $D\ge {n\choose m+2}-\sum_{j=0}^m{n\choose j}$. On the other hand, we have ${n\choose m+2}\ge {n\choose m+1}$ since $n=s(m+1)-1\ge 2m+3$ for $m\ge 1$ and $s\ge 3$. Thus, $D\ge {n\choose m+1}-\sum_{j=0}^m {n\choose j}> 0$.
\end{proof}

From now on, we assume that $\ell \ge 2$. We use the same idea, however, we will average over a slightly different collection of sets. Take $s$ pairwise disjoint sets $H_1,\ldots,H_s$ of size $m-1$ at random. W.l.o.g. assume that the $2s-\ell$ elements of $[n]\setminus\bigcup_{i=1}^sH_i$ form the set $[2s-\ell]$.


For all $i \in [s]$, $0\le j\le m-1$ and $S\subset [2s-\ell]$, define the random variables $\beta_i^{(j)}$ and $\beta_i(S)$ analogously to how it is done in \eqref{eq50v1}. Note that the analogues of \eqref{eq45} and \eqref{eq46v1} hold in this case as well.

The analogue of Lemma~\ref{lem11v1} in this case is the following statement. Once verified, the rest of the argument is the same.
\begin{lem}\label{lem11} For every choice of $H_1,\ldots, H_s$ and the full chains one has
\begin{equation}\label{eq47}\sum_{i=1}^s\Biggl[\sum_{j=1}^{m-2}{n\choose j}\beta_i^{(j)}+\sum_{S\subsetneq[2s-\ell],|S|\le 3}\frac{{n\choose m-1+|S|}}{{2s-\ell\choose |S|}}\beta_i(S)\Biggr]\le s\sum_{j=m+1}^{m+2}{n\choose j}+(\ell-1){n\choose m}.
\end{equation}
\end{lem}

The proof of Lemma~\ref{lem11} uses the following proposition.
\begin{prop}\label{prop2} For $n' := 2s'-\ell'$ with $\ell\in [s']$ and for $s'$ cross-dependent families $\ff'_1,\ldots,\ff'_{s'}\subset {[n']\choose 1}\cup {[n']\choose 2}$ we have
\begin{equation}\label{eq49}\sum_{i=1}^{s'}|\ff'_i|\le (\ell'-1)n'+s'{n'\choose 2}.\end{equation}\end{prop}

\begin{proof}[Proof of Proposition \ref{prop2}]
Fix a random ordering on $[n']$. For $S\subset [n'],$ put $\vartheta_i(S)=1$ if $S\in \ff_i'$ and $\vartheta_i(S)=0$ otherwise.
Similarly to \eqref{eq47v1}, (\ref{eq47}), it is sufficient to prove
\begin{equation}\label{eq51}\sum_{i=1}^{s'}\Biggl[n'\vartheta_i(\{i\}) +\sum_{x=s'+1}^{n'}\frac{{n'\choose 2}}{n'-s'}\vartheta_i(\{i,x\})\Biggr]\le (\ell'-1)n'+s'{n'\choose 2}.\end{equation}
W.l.o.g., suppose that $\vartheta_i(\{i\})=1$ if and only if $i\in[p]$.  If $p\le \ell'-1$, then we are done. If $\ell'=s'$ then the statement of the proposition is obvious since at least one of $\vartheta_i(\{i\})$ is equal to 0. Thus, we assume that $\ell'\le p\le s'-1$.

 Recall that $n' = s'+(s'-\ell')$. For any collection of distinct elements $x_{p+1},\ldots, x_{s'}\in [s'+1,n']$ we have $\sum_{i=p+1}^{s'}\vartheta_i(\{i,x_i\})\le s'-p-1$. By a simple averaging argument we get that at least a $1/(s'-p)$-proportion of pairs $(i,x)$, where $i =p+1,\ldots, s'$ and $x=s'+1,\ldots, n'$, satisfies $\vartheta_i(\{i,x\}) = 0$. This accounts for at least $n'-s' = s'-\ell'$ such pairs.
Therefore, the left hand side of (\ref{eq51}) in this case does not exceed $(s'-1)\Bigr[{n'\choose 2}+n'\Bigl]$, which is at most the right hand side of (\ref{eq51}) since ${n'\choose 2}=n'(s'-\frac{\ell'+1}2)\ge(s'-\ell')n'$.
\end{proof}

\begin{proof}[Proof of Lemma \ref{lem11}]

 We have $\pmb{\sum_{i=1}^s\beta_i(\emptyset)=p}$ for some $0\le p\le s-1$.
 W.l.o.g. assume that $\beta_i(\emptyset)=1$ if and only if $i\in [p]$.


Assume that $\pmb{ \ell/2\le p\le s-2}$. In this case we have  $s-p\le s-\ell/2 = \frac 12(2s-\ell)$. Then $\sum_{i=p+1}^s\beta_{i}(S_i)\le s-p-1$ for any pairwise disjoint $S_{p+1},\ldots, S_s$ of cardinality two. By a simple averaging argument we immediately get that at least a $1/(s-p)$-proportion of all pairs $(i,S)$ satisfy $\beta_i(S)=0$, where $i=p+1,\ldots, s$ and $S\in {[2s-\ell]\choose 2}$. In other words,
\begin{equation}\label{eqstup0}\beta_i(S) =0\text{ for at least }{2s-\ell\choose 2}\text{ pairs }(i,S),\text{ where }|S| \in {[2s-\ell]\choose 2}\text{ and }i\in [p+1,s].\end{equation}
Since the families $\ff_1,\ldots, \ff_s$ are cross-dependent, we similarly get that \begin{equation}\label{eqstup}\beta_i(S) = 0\text{ for at least \ \ \ }2s-\ell\text{\ \ \ pairs }(i,S),\text{ where }|S| \in {[2s-\ell]\choose 1} \text{ and }i\in [s].\ \ \ \ \ \ \ \ \end{equation}

  Therefore, the left hand side of (\ref{eq47}) is at most $p\sum_{j=0}^{m-1}{n\choose j}+(s-1)\sum_{j=m}^{m+1}{n\choose j}+s{n\choose m+2}$. Consider the difference $D$ between the RHS and the LHS of (\ref{eq47}). Then $D$ is at least
 \begin{equation}\label{eq52}{n\choose m+1}-(s-\ell){n\choose m}-(s-2)\sum_{j=0}^{m-1}{n\choose j}.\end{equation}
 Next we show that this expression is always nonnegative.
We have
\begin{equation}\label{eq54} {n\choose m+1}-(s-\ell){n\choose m} = \Bigl(\frac{m(s-1)+s-\ell}{m+1}-(s-\ell)\Bigr){n\choose m} =\frac{m(\ell-1)}{m+1}{n\choose m}.\end{equation}
It is easy to see that the right hand side of (\ref{eq54}) is bigger than the right hand side of (\ref{eq55}) for both $\ell=2$ and $\ell\ge 3$, which proves (\ref{eq47}) in this case. \vskip+0.1cm

If $\pmb{p=s-1}$, then $\beta_s(S)=0$ for all $S\subset [2s-\ell], |S|\le 3$. Therefore, the LHS (\ref{eq47}) is at most $(s-1)\sum_{j=0}^{m+2}{n\choose j}$, and $D$ is at least \begin{equation}\label{eq56}{n\choose m+2}+{n\choose m+1}-(s-\ell){n\choose m}-(s-1)\sum_{j=0}^{m-1}{n\choose j}.\end{equation}
Recall that $m\ge 1$. If $\ell\ge 3$, then one can see from \eqref{eq54} and \eqref{eq55} that
$${n\choose m+1}-(s-\ell){n\choose m}-(s-1)\sum_{j=0}^{m-1}{n\choose j}\ge \Big(\frac {m(\ell-1)}{m+1}-\frac{s-1}{s-2}\cdot \frac{(s-1)m-1}{(s-1)m+s-\ell+1}\Big){n\choose m},$$
which is always nonnegative. Indeed, it is easy to see for $\ell\ge 4$ (and thus $s\ge 4$), and for $\ell=3$ it can be verified separately.

If $\ell=2$ then due to \eqref{eq54} and \eqref{eq55} the expression \eqref{eq56} is at least ${n\choose m+2}-\sum_{j=0}^{m-1}{n\choose j}$. We have ${n\choose m+2}\ge {n\choose m}$ since $n = s(m+1)-\ell\ge sm+1\ge 2m+2$ for any $s\ge 3$ and $m\ge 1$. Finally, ${n\choose m}>\sum_{j=0}^{m-1}{n\choose j}$ by (\ref{eq55}). \vskip+0.1cm

 Suppose that $\pmb{p<\ell/2}$. Then  again  $\prod_{i=p+1}^s\beta_{i}(S_i)=0$ for any $s-p$ pairwise disjoint $S_i$. For each $i\in [p+1,s]$, consider the family $\ff'_i := \{S\subset[2s-\ell]: \beta_i(S) = 1, |S|\le 2\}$. These families are cross-dependent.  Applying (\ref{eq49}) to $\ff'_i$ with $2(s-p)-(\ell-2p),\ s-p$ and $\ell-2p$ playing the roles of $n',\ s'$ and $\ell'$, respectively, we get that
$$\sum_{i=p+1}^s|\ff'_i|\le (\ell-2p-1)(2s-\ell)+(s-p){2s-\ell\choose 2}.$$
(Note that we tacitly use that none of the $\ff_{p+1}',\ldots, \ff'_s$ contain the empty set by the assumption.)

We conclude that among the coefficients $\beta_i(S)$, where $i\in [s]$ and $|S|\in [2]$, there are at least $(2s-\ell)(s-\ell+p+1)$ that are equal to zero. The following observation is verified by a simple calculation.

\begin{obs}\label{obs11} Suppose $s\ge 3$ and $m\ge 1$. In the summation over $S$ in (\ref{eq47}), the coefficient in front of $\beta_i(S_1)$ for $|S_1| = 1$  is not bigger than the coefficient in front of $\beta_i(S_2)$ for $|S_2| = 2$. That is,
$$\frac{{n\choose m+1}}{{2s-\ell\choose 2}}\ge \frac{{n\choose m}}{{2s-\ell\choose 1}}.$$
\end{obs}
Using the observation and the conclusion of the paragraph above it, we get that the left hand side of (\ref{eq47}) is at most $$p\sum_{j=0}^{m-1}{n\choose j}+(\ell-p-1){n\choose m}+s{n\choose m+1}+s{n\choose m+2}.$$ Since ${n\choose m}> \sum_{j=0}^{m-1}{n\choose j}$, the last expression is smaller than $(\ell-1){n\choose m}+s{n\choose m+1}+s{n\choose m+2}.$
\end{proof}


\subsection{Proof of Theorem \ref{thm7}}

We prove the theorem by double induction. We apply induction on $m$, and for fixed $m$ we apply induction on $n$. The case $m = 0$ of (\ref{eq019}) is very easy to verify. The case $n = q$ is the bound (\ref{eqcross}).

We may assume that all the $\ff_i$ are shifted. The following two families on $[n-1]$ are typically defined for a family  $\ff\subset 2^{[n]}$:
\begin{align*} &\ff(n) := \bigl\{A-\{n\}: n\in A, A\in \ff\bigr\},\\
&\ff(\bar n):= \bigl\{A: n\notin A, A\in \ff\bigr\}.
\end{align*}

It is clear that $\ff_1(\bar n),\ldots, \ff_s(\bar n)$ are $q$-dependent. Next we show that $\ff_1(n),\ldots, \ff_s(n)$ are $(q-s)$-dependent. Assume for contradiction that $F_1,\ldots, F_s,$ where $F_i\in \ff_i(n),$ are pairwise disjoint and that $H:= F_1\cup\ldots\cup F_s$ has size at most $q-s$. Since $n\ge q$, we have $n-(q-s)\ge s$. Therefore, we can find distinct elements $x_1,\ldots, x_s\in [n]-H.$ Since $\ff_i$ are shifted, we have $F_i\cup\{x_i\}\in \ff_i$ for $ i\in[s]$, and the sets $F_i\cup\{x_i\}$ are pairwise disjoint. Their union $H\cup\{x_1,\ldots, x_s\}$ has size $|H|+s\le q$,  contradicting the assumptions of the theorem.

Recall the definition (\ref{eq085}). The induction hypothesis for $\ff_i(\bar n)$ gives
\begin{equation}\label{eq020}\sum_{i=1}^s|\ff_i(\bar n)|\le \sum_{i=1}^s|\tilde{\ff}_i^{n-1,q}|.\end{equation}

The induction hypothesis applied to $\ff_i(n)$ with $n-1,q-s, m-1$ playing the roles of $n,q, m$ gives
\begin{equation}\label{eq021}\sum_{i=1}^s|\ff_i(n)|\le \sum_{i=1}^s|\tilde{\ff}_i^{n-1,q-s}|.\end{equation}
Adding up (\ref{eq020}) and (\ref{eq021}), we get
$$\sum_{i=1}^s|\ff_i| = \sum_{i=1}^s\bigl(|\ff_i(n)|+|\ff_i(\bar n)|\bigr)\le \sum_{i=1}^s\bigl(|\tilde{\ff}_i^{n-1,q}|+|\tilde{\ff}_i^{n-1,q-s}|\bigr).$$
We have $\tilde{\ff}_i^{n-1,q} = \tilde{\ff}_i^{n,q}(\bar n)$ and $\tilde{\ff}_i^{n-1,q-s} = \tilde{\ff}_i^{n,q}(n).$ Thus, for any $i\in [s]$
$$|\tilde{\ff}_i^{n-1,q}|+|\tilde{\ff}_i^{n-1,q-s}| = |\tilde{\ff}_i^{n,q}(\bar n)|+|\tilde{\ff}_i^{n,q}(n)|=|\tilde{\ff}_i^{n,q}|.$$

\section{Application of Theorem \ref{thmhil} to an anti-Ramsey problem}
We call a partition $\ff_1\sqcup\ldots\sqcup \ff_{M}$ of
${[n]\choose k}$ into non-empty families $\ff_i$ an {\it $M$-coloring}. Let the anti-Ramsey number $ar(n,k,s)$ be the minimum $M$ such that in any $M$-coloring there is a \textit{rainbow $s$-matching}, that is, a set of $s$ pairwise disjoint $k$-sets from pairwise distinct $\ff_i$.

This quantity was studied by \"Ozkahya and Young \cite{OY}, who have made the following conjecture.
\begin{gypo}[\cite{OY}]\label{conoy} One has $ar(n,k,s) =e_k(n,s-1)+2$ for all $n> sk$.
\end{gypo}
It is not difficult to see that $ar(n,k,s)\ge e_k(n,s-1)+2$ for any $n,k,s$. Indeed, consider a maximum size family of $k$-sets with no $(s-1)$-matching and assign a different color to each of these sets. Next, assign one new color to all the remaining sets. This is a coloring of ${[n]\choose k}$ without a rainbow $s$-matching.
In \cite{OY} the authors proved Conjecture~\ref{conoy} for $s=3$ and for $n\ge 2k^3s$. They also obtained the bound $ar(n,k,s)\le e_k(n,s-1)+s$ for $n\ge sk+(s-1)(k-1)$.

In this section we first state and prove a result for $n\ge sk+(s-1)(k-1)$, which is much stronger than Conjecture \ref{conoy} in that range. We say that the $M$-coloring of ${[n]\choose k}$ is \textit{$s$-star-like} if there exists a set $Y\in {[n]\choose s-2}$ and a number $i\in [M]$ such that for every $j\in [M]-\{i\}$ and $F\in \ff_j,$ $F$ intersects $Y$. Clearly, each $s$-star-like coloring has at most $e_k(n,s-1)+1$ colors. For convenience, we define the quantity $$h(n,k,s):=\max\{|\ff|:\ff\subset {[n]\choose k},\nu(\ff)<s,\tau(\ff)\ge s\},$$ which was determined for a certain range in Theorem \ref{thmhil}.

\begin{thm}\label{thmar} Let $s\ge 3$, $k\ge 2$ and $n\ge sk+(s-1)(k-1)$ be some integers. Consider an $M$-coloring of ${[n]\choose k}$ without a rainbow $s$-matching. Then either this coloring is $s$-star-like, or $M\le h(n,k,s-1)+s$.
\end{thm}

Conjecture \ref{conoy} follows from Theorem \ref{thmar}, once we can apply Theorem \ref{thmhil} or an analogous statement. Indeed, it is shown in Theorem \ref{thmhil} that $h(n,k,s-1)$ is much smaller than $e_k(n,s-1)$, moreover, it implies that in most cases  $h(n,k,s-1)+s$ is smaller than $e_k(n,s-1)+2$. We do not need an exact Hilton--Milner-type result to deduce Conjecture~\ref{conoy}. We may use a weaker form of Theorem \ref{thmhil} that was proven in \cite{FK7} and is valid for a slightly wider range of parameters than Theorem \ref{thmhil}. In particular, \cite[Theorem~5]{FK7} implies that for $n\ge sk+(s-1)(k-1)$ we have $e_k(n,s-1)-h(n,k,s-1)\ge \frac 1s {n-s-k+2\choose k-1}$. For any $n\ge sk+(s-1)(k-1)$ and $k\ge 3$ we have $\frac 1s {n-s-k+2\choose k-1}> s-2$. Thus, Theorem~\ref{thmar} implies Conjecture \ref{conoy} for this range.

\begin{cor} We have $ar(n,k,s) = e_k(n,s-1)+2$ for $n\ge sk+(s-1)(k-1)$ and $k\ge 3$.
\end{cor}
 Note that this is the same range in which \"Ozkahya and Young got a weaker bound $ar(n,k,s) \le e_k(n,s-1)+s$. We remark that the case $k = 2$ has already been settled for all values of parameters (see \cite{OY} for the history of the problem).

\begin{proof}[Proof of Theorem \ref{thmar}]
Arguing indirectly, fix an $M$-coloring with $M\ge h(n,k,s-1)+s$  and with no rainbow $s$-matching.
We may assume that there is a rainbow $(s-1)$-matching $F_1,\ldots, F_{s-1}$ with, say, $F_i\in \ff_{M-i+1}$ for each $i\in [s-1]$. Otherwise, recolor some of the sets in new colors so that the number of colors increases and a rainbow $(s-1)$-matching (but no rainbow $s$-matching) appears.

For each $i\in [M-s+1]$, choose $G_i\in \ff_i$. Note that $M-s+1> h(n,k,s-1)$. Thus, either $\mathcal G:=\{G_1,\ldots, G_{M-s+1}\}\subset {[n]\choose k}-{U\choose k}$ for a suitable $U\subset {[n]\choose n-s+2}$ and for any choice of $G_i$, or for some choice of $G_i$ there is an $(s-1)$-matching, say $G_1,\ldots, G_{s-1}$, in $\mathcal G$.

In the latter case we can apply the argument used by \"Ozkahya and Young: since the colors of $F_1,\ldots, F_{s-1},G_1,\ldots, G_{s-1}$ are all distinct, any $k$-set from $[n]\setminus \bigcup_{i=1}^{s-1}G_i\cup F_i$ forms a rainbow $s$-matching with one of these two $(s-1)$-matchings. Moreover,  $G_i\cap (F_1\cup\ldots\cup F_{s-1})\ne \emptyset$ holds for each $i\in[s-1]$. Thus, $|\bigcup_{i=1}^{s-1}G_i\cup F_i|\le (2k-1)(s-1)$, and we are done provided $|[n]\setminus \bigcup_{i=1}^{s-1}G_i\cup F_i|\ge k$, which holds for $n\ge k+(2k-1)(s-1) = sk+(s-1)(k-1)$.

In the former case the family $\mathcal G$ must satisfy $\tau(\mathcal G)\le s-2$ \textit{for all choices} of  $G_1\in \ff_1,\ldots, G_{M-s+1}\in\ff_{M-s+1}$.  

\begin{cla}\label{cla1} Fix $N> h(n,k,s-1)$ and pairwise disjoint families $\mathcal H_1,\ldots, \mathcal H_N$ of $k$-subsets of $[n]$. If for any set of representatives $\mathcal H:=\{H_1,\ldots, H_N\}$ with $H_i\in \mathcal H_i$  we have $\tau(\mathcal H)\le s-2$, then $\tau\big(\bigcup_{i=1}^{N}\mathcal H_i\big)\le s-2$.
\end{cla}
\begin{proof}Fix one choice of $\mathcal H$ and let $T$ be a hitting set of size $s-2$ for $\mathcal H$. Arguing indirectly, assume that there is a set $H'\in \mathcal H_1$ such that $H'\cap T = \emptyset$. The family $\mathcal H':= \{H',H_2,\ldots, H_{N}\}$ also satisfies $\tau(\mathcal H')= s-2$, and so there is a set $T'\ne T$, $|T'|\le s-2$, such that $H_2,\ldots, H_{N}$ all intersect $T'$.
Define $m(T,T'):=\bigl|\{F\subset {[n]\choose k}: F\cap T\ne \emptyset\ne F\cap T'\}\bigr|.$
We want to show that for any distinct $T,T'$ of size $s-2$ the quantity $m(T,T')$ is strictly smaller than $h(n,k,s-1).$ This contradicts the fact that $H_2,\ldots, H_N$ all intersect $T$ and $T'$.

Let us show that $m(T,T')$ is maximal when $|T\cap T'| = |T|-1 =s-3$. Indeed if $|T\cap T'|<|T|-1$, then we may choose $x\in T\setminus T', y\in T'\setminus T$ and define $T'':=(T'-\{y\})\cup \{x\}$. Let $F$ be an arbitrary set satisfying $F\cap T\ne \emptyset $, $F\cap T'\ne \emptyset$ and  $F\cap T''=\emptyset$. This means that $F\cap (T'\cup T'') = \{y\}$, $F\cap ((T-\{x\})\setminus T')\ne \emptyset$.
Setting $|T\cap T'| = t$, the number of such sets $F\in {[n]\choose k}$ is ${n-(s-2)-1\choose k-1}-{n-2(s-2)+t\choose k-1}$.

On the other hand, the sets $F$ satisfying $F\cap T\ne \emptyset, F\cap T''\ne \emptyset$, and $F\cap T' = \emptyset$ are those with $F\cap (T'\cup T'') = \{x\}$. Their number is ${n-(s-2)-1\choose k-1}$, which is clearly bigger.

Suppose now that $|T\cap T'|=s-3$. Then  $m(T,T') = {n\choose k}-{n-s+3\choose k}+{n-s+1\choose k-2}$. Since $h(n,k,s-1)\ge {n\choose k}-{n-s+2\choose k}-{n-s+2-k\choose k-1}+1$, we have $h(n,k,s-1)-m(T,T')> {n-s+1\choose k-1}-{n-s+2-k\choose k-1}> 0$. This completes the proof of the claim.
\end{proof}

Applying Claim \ref{cla1} to the first $M-s+1$ color classes, we get that they all intersect a set $T$ of size $s-2$. To complete the proof, we need to show that the same holds for some $M-1$ colors.

Note that since $\sum_{i=1}^{M-s+1}|\ff_i|\le {n\choose k}-{n-s+2\choose k}$, one of the last $s-1$ color classes, say $\ff_{M}$, has size at least $\frac 1{s-1}{n-s+2\choose k}$.

\begin{cla}In every rainbow $(s-1)$-matching, one of the $k$-sets belongs to $\ff_{M}$.\end{cla}
\begin{proof} Arguing indirectly, assume that there is a rainbow  $(s-1)$-matching in colors $\ff_{M-s+1},\ldots,$ $\ff_{M-1}$. Applying Claim \ref{cla1} to $\ff_1,\ldots, \ff_{M-s},\ff_{M}$, we find a hitting set $T$ of size $s-2$ for $\ff_1\cup\ldots\cup \ff_{M-s}\cup\ff_{M}$. We infer
\begin{equation}\label{eq099}M-s+\frac 1{s-1}{n-s+2\choose k}\le |\ff_1\cup\ldots\cup \ff_{M-s}\cup\ff_{M}| \le {n\choose k}- {n-s+2\choose k}.\end{equation}
We have $M-s\ge h(n,k,s-1)>{n\choose k}-{n-s+3\choose k}$. Also, we have $${n-s+3\choose k} = \frac{n-s+3}{n-s+3-k}{n-s+2\choose k}\le \frac {s+1}s{n-s+2\choose k},$$
provided $n-s+3-k\ge sk$. The inequalities above contradict (\ref{eq099}), and so the claim follows.\end{proof}

We conclude that there is no rainbow $(s-1)$-matching in $\ff_1\cup\ldots\cup\ff_{M-1}$, which implies that there is a cover of size $s-2$ for any set of distinct representatives of the color classes. In turn, Claim \ref{cla1} implies that $\ff_1\cup\ldots\cup\ff_{M-1}$ can be covered by a set $T$ of size $s-2$, i.e., the coloring is $s$-star-like. Theorem \ref{thmar} is proved.\end{proof}

Next, we prove a weaker bound on $ar(n,k,s)$, which is valid for all $n> sk$:

\begin{thm} We have $ar(n,k,s)\le e_k(n,s-1)+\frac{(s-1){n\choose k}}{{n-(s-1)k\choose k}}+1$ for any $n>ks$.
\end{thm}

\begin{proof}
Fix an $M$-coloring of ${[n]\choose k}$ with no rainbow $s$-matching and let $G_1,\ldots, G_M$ be a set of representatives of the color classes. Reordering $G_i$ if necessary, we may assume that for some integer $T$ the  collections $\{G_1,\ldots, G_{s-1}\},\ \{G_s,\ldots, G_{2(s-1)}\},$ $\ldots,\{G_{(T-1)(s-1)+1},$ $\ldots, G_{T(s-1)}\}$ are rainbow $(s-1)$-matchings, while the remaining collection $\{G_{T(s-1)+1},\ldots, G_M\}$ does not contain a rainbow $(s-1)$-matching. We have \begin{equation}\label{eqar}M-T(s-1)\le e_k(n,s-1).\end{equation}

Set $H_i := [n]-(G_{i(s-1)+1}\cup\ldots\cup G_{(i+1)(s-1)})$ for $i=0,\ldots, T-1$. The non-existence of rainbow $s$-matchings in the coloring implies that in the coloring of ${H_i\choose k}$ we only used the colors of the $s-1$ sets $G_{i(s-1)+1},\ldots, G_{(i+1)(s-1)}.$ Consequently, ${H_i\choose k}$ and ${H_j\choose k}$ must be disjoint for $i\ne j$. From this ${n\choose k}\ge T{n-(s-1)k\choose k}$ follows. Combining with (\ref{eqar}), we obtain $M\le e_k(n,s-1)+(s-1){n\choose k}/{n-(s-1)k\choose k}$, as claimed.\end{proof}

Finally, we prove a bound that is valid for $n>(s+\sqrt s) k$ and is stronger than the previous one in many cases.

\begin{thm} We have  $ar(n,k,s)< e_k(n,s-1)+\frac{(s-1)n(n-sk)}{(n-sk)^2-k((2s-1)k-n)}+1$ for any $n>(s+\sqrt s)k$. In particular, if $n = csk$ for some fixed $c>1$ and $s\to \infty$, then $ar(n,k,s)\le e_k(n,s-1)+(1+o(1))s\frac{c}{c-1}$.
\end{thm}

\begin{proof} Following the notations of the previous proof, put $F_i:=G_{i(s-1)+1}\cup\ldots\cup G_{(i+1)(s-1)}$ for $i=0,\ldots, T-1$. Note that $|F_i| = (s-1)k$. Then for each $i\ne j$ we have $|F_i\cup F_j|>n-k$ since otherwise any $k$-set in $[n]\setminus (F_i\cup F_j)$ will form a rainbow $s$-matching  with either the sets forming $F_i$ or the sets forming $F_j$. Thus, we have $|F_i\cap F_j|=2|F_i|-|F_i\cup F_j| < (2s-1)k-n$ and
$$\sum_{0\le i<j\le T-1}|F_i\cap F_j|< {T\choose 2}((2s-1)k-n).$$
On the other hand, if an element $x$ is contained in $d_x$ sets $F_i$, then it contributes ${d_x\choose 2}$  to the sum $\sum_{0\le i<j\le T-1}|F_i\cap F_j|$. The average number $d$ of sets containing an element of $[n]$   satisfies $d= \frac {T(s-1)k}n$. Therefore, we get
$$\sum_{0\le i<j\le T-1}|F_i\cap F_j| = \sum_{x\in [n]}{d_x\choose 2}\ge n{d\choose 2}.$$
Combining the two displayed formulas, we get
$$T(T-1)((2s-1)k-n))> n\frac{T(s-1)k}n\Big(\frac{T(s-1)k}n-1\Big).$$
After rearranging and simplifying the expression above we get $$\Big(\frac{((s-1)k)^2}n-(2s-1)k+n\Big)T< (s-1)k-(2s-1)k+n,$$
which is equivalent to
\begin{equation}\label{eqarr} \big((n-sk)^2-(2s-1)k^2+nk\big)T< n(n-sk).\end{equation}
The left hand side is positive provided $n=sk+x$, where $x^2\ge k((s-1)k-x)$. This holds for $x = k\sqrt s$. In this assumption it is clear that (\ref{eqarr}) combined with \eqref{eqar} implies the first statement of the theorem. If $s\to \infty$ and $n = csk$ with $c>1$, then (\ref{eqarr}) transforms into $(1+o(1))(c-1)^2T\le c(c-1)$, which, together with \eqref{eqar}, implies the second statement of the theorem.
\end{proof}

We remark that a similar proof, applied in a more general scenario, appeared in \cite{Kup}.

\section{Almost matchings}

Let us say that the sets $F_1,\ldots, F_s$ form an \textit{almost matching} if the family $\ff:=\{F_1,\ldots, F_s\}$ has at most one vertex of degree greater than one and that vertex has degree at most two.

Define $$a(n,s):=\max\bigl\{|\ff|: \ff\subset 2^{[n]}, \ff\text{ contains no almost-matching of size }s\bigr\}.$$


\begin{thm}\label{thm12} The equality $a(sm-2,s) = \sum_{t= m}^{sm-2}{sm-2\choose t}$ holds for all $s\ge 2, m\ge 1.$
Moreover, the equality is achieved only by the family $\{F\subset [sm-2]: |F|\ge m\}$.
\end{thm}
\begin{proof} Suppose $\ff\subset 2^{[n]}$ contains no almost-matching of size $s$, where $n=sm-2$.
We may suppose that $\ff$ is an up-set.  Consider the families $\ff_1,\ldots, \ff_s$, where $\ff_{i} := \ff$ for $i=1,\ldots, s-1$, and $\ff_s := \partial \ff\cup\{[n]\}$ (see (\ref{eq86})).

  \begin{cla}
    The families $\ff_1,\ldots, \ff_s$ are cross-dependent.
  \end{cla}
  \begin{proof}
    Indeed, if $F_1,\ldots, F_s$, where $F_i\in \ff_i$, are pairwise disjoint then, replacing $F_s$ by some $F\in \ff$, $F_s\subset F, |F\setminus F_s|= 1$, we obtain $s$ members $F_1,\ldots, F_{s-1},F$ of $\ff$ that form an almost-matching.
  \end{proof}
  Applying Theorem \ref{thm2} yields
  \begin{equation}\label{eq87}\sum_{i=1}^s|\mathcal F_i|\le {n\choose m-1}+s\sum_{t= m}^{n}{n\choose t}.\end{equation}

\begin{lem}\label{lem21}
  Let $\ff\subset 2^{[n]}$ be an up-set and suppose that for some $1\le k\le n$, $|\ff|=\sum_{i=k}^n {n\choose i}$. Then $\partial \ff$ satisfies \begin{equation}\label{eqsh1}
                   \partial(\ff)\ge \sum_{i=k-1}^{n-1}{n\choose i}
                 \end{equation}
  with equality if and only if $\ff={[n]\choose \ge k}$.
\end{lem}
Let us first deduce Theorem~\ref{thm12} from Lemma~\ref{lem21}. By the lemma, if $|\ff|>\sum_{t=m}^n {n\choose t}$ or $|\ff|=\sum_{t=m}^n {n\choose t}$ and $\ff\ne {[n]\choose \ge m}$   then $|\ff_s|>\sum_{t=m-1}^n {n\choose t},$ which contradicts \eqref{eq87}. This yields the statement of the theorem.
\end{proof}

\begin{proof}[Proof of Lemma~\ref{lem21}] We apply induction on $n$. The cases $n=1,2$ are trivial, since ${[n]\choose \ge k}$ is the only up-set of the given size. For $n=3,k=1$ we have one more shifted up-set, namely $\{P:\{1\}\subset P\subset [3]\}$. For this family one has strict inequality in \eqref{eqsh1}.

Suppose now that \eqref{eqsh1} is true for $n$ and all $k$ and let us prove it for $n+1$ and all $k$. Note that the cases $k=1$, $k=n+1$ are obvious and suppose that $1<k\le n$. First suppose that $\ff$ is shifted.
\begin{cla}
  $|\ff(1)|\ge \sum_{i=k}^{n+1}{n\choose i-1}.$
\end{cla}
\begin{proof} The opposite would imply $|\ff(\bar 1)|>\sum_{i=k}^{n+1}{n+1\choose i}-{n\choose i-1}=\sum_{i=k}^n{n\choose i}$.
  By the induction hypothesis, $\partial\ff(\bar 1)\ge \sum_{i=k-1}^{n-1}{n\choose i}=\sum_{j=k}^{n+1}{n\choose j-1}-1$.
By shiftedness, if $G\in \partial \ff(\bar 1)$ then $G\cup \{1\}\in \ff$, i.e., $G\in \ff(1)$. Noting that $[n+1]\in\ff$ implies that $[2,n+1]\in\ff(1)$ as well, the claim follows.
\end{proof}
Applying the induction hypothesis for $\ff(1)$ yields $$|\partial \ff(1)|\ge \sum_{i=k}^{n+1}{n\choose i-2}.$$
Together with the claim and the fact that $\partial \ff\supset \ff(1)$ we infer
$$|\partial\ff|\ge \sum_{i=k}^{n+1}{n\choose i-1}+{n\choose i-2}=\sum_{i=k}^{n+1}{n+1\choose i-1},$$
equivalent to the bound in Lemma~\ref{lem21}.

In case of equality the induction hypothesis implies $\ff(1)={[n]\choose \ge k-1}$. By shiftedness, $|\ff|\ge k$ for {\it all} $F\in\ff$. This in turn implies $\ff\subset {[n+1]\choose \ge k}$.

Finally, we note that shifting does not change the size of the sets. Thus, $S_{ij}(\ff)={[n+1]\choose \ge k}$ can only occur if already $\ff={[n+1]\choose \ge k}$. Therefore, the lemma is true for not necessarily shifted families as well.
\end{proof}

\section{Conclusion} In this paper we have obtained several results related to families of sets with no $s$ pairwise disjoint sets. The stability results in the spirit of Theorem \ref{thmhil} have proven to be useful. We have found two applications so far: to Conjecture \ref{conoy} and to Erd\H os and Kleitman's problem on families with no matchings of size $s$, studied in \cite{FK7}.

The method we developed for the proof of Theorem \ref{thm2} may be applied in different scenarios. Apart from the result on almost matchings, we have applied a modification of this method in \cite{FK8} to determine the size of the largest families with no matchings of size 3 and 4, as well as to an old problem concerning families with no partitions \cite{FK10}. Similar ideas also appeared in \cite{FK13}.
\section{Acknowledgements}
We thank the referees for carefully reading the text and pointing out numerous inaccuracies in it. Thanks to their comments, the presentation of the paper has greatly improved.

\end{document}